 \documentclass{amsart}
\usepackage{graphicx}
\usepackage{amsthm,amsmath,verbatim,amssymb}
\usepackage{subcaption}
\usepackage{caption}
\usepackage{color}
\usepackage{arydshln}

\numberwithin{equation}{section}
\theoremstyle{plain}
\def\definedas{\stackrel{\Delta}{=}}

\def\:{:\,}\numberwithin{equation}{section}
\allowdisplaybreaks
\newtheorem{theorem}{Theorem}[section]
\newtheorem{lemma}[theorem]{Lemma}
\newtheorem{proposition}[theorem]{Proposition}
\newtheorem{corollary}[theorem]{Corollary}

\renewcommand{\d}{\partial}

\newcommand{\beq}{\begin{eqnarray}}
\newcommand{\eeq}{\end{eqnarray}}
\newcommand{\beqq}{\begin{eqnarray*}}
\newcommand{\eeqq}{\end{eqnarray*}}
\newcommand{\bc}{\begin{center}}
\newcommand{\ec}{\end{center}}
\def\real{{\mathbb{R}}}
\newcommand{\lips}{{\mathcal L}}

\def\definedas{\stackrel{\Delta}{=}}

\def\E{\mathbb E}
\def\And{A_n^d}
\def\Emnd{\mathbb E_{\mu_n^d}}
\def\Pmnd{\mathbb P_{\mu_n^d}}

\def\wN{\widehat N}
\def\Pdn{{\Phi^d_n}}
\def\kdn{{k^d_n}}

\def\Hdn{{\mathcal H^d_n}}
\def\smallhalf{\mbox{ $\frac{1}{2}$}}

\begin{document}
\title[Random spherical harmonics]{Critical radius and supremum of random spherical harmonics}
\author[Feng]{Renjie Feng}
\address{Beijing International Center for Mathematical Research, Peking University, Beijing, China}
\email{renjie@math.pku.edu.cn}

\author[Adler]{Robert J. Adler}
\address{Andrew and Erna Viterbi Faculty of Electrical Engineering
Haifa 32000, Israel}
\email{robert@ee.technion.ac.il}

\maketitle
\begin{abstract}

We first consider  {\it deterministic} immersions 
 of the  
$d$-dimensional sphere into high dimensional Euclidean spaces, where the immersion is via spherical harmonics of level $n$. The main result of the article is the, a priori unexpected, fact that there is a uniform lower bound to the critical radius of the immersions as $n\to\infty$.  This fact has immediate implications for {\it random} spherical harmonics  with  fixed $L^2$-norm. In particular,  it leads to an exact and explicit formulae for the tail probability of their  (large deviation) suprema  by the tube formula, and  also relates this to the expected Euler characteristic of their upper level sets.
\end{abstract}

%\begin{keyword}[class=MSC]
%\kwd[Primary ]{33C55, 60G15}
%\kwd[; secondary ]{60F10, 60G60.}
%\end{keyword}

%\begin{keyword}
%\kwd{Spherical harmonics, spherical ensemble, Gaussian ensemble,  critical radius, reach, curvature, asymptotics, large deviations.}
%\end{keyword}

\section{Introduction}

The spherical harmonics, of level $n\geq 1$, on the $d$-dimensional unit sphere $S^d$, are the collection of the 
\begin{equation}
\label{dimn}
k^d_n\ =\ \frac{2n+d-1}{n+d-1}{n+d-1\choose d-1} 
\end{equation}
eigenfunctions $\{\phi_j^{n,d}\}_{j=1}^\kdn$ of the Laplacian $\Delta_d$ on $S^d$, satisfying,
\beq
\label{eq:eigen}
\Delta_d \phi^{n,d}_j(x) =-n(n+d-1)\phi_j^{n,d}(x).
\eeq
It is then immediate that for any vector $a=(a_1,\dots,a_\kdn)$ of reals, the  functions 
\beq
\label{eq:rsh}
\Phi^d_n\  \definedas  \ \sum_{j=1}^{k^d_n} a_j\phi_j^{n,d} 
\eeq
solve the wave equation 
\beq
\label{eq:wave}
\Delta_d \Pdn\ = \ \alpha\, \Pdn,
\eeq
where
$\alpha=-n(n+d-1)$. Thus, with some ambiguity, both the $\Pdn$ and their linear combinations are often also referred to as spherical harmonics, or wave functions on the sphere.

Instead of taking the $a_j$ in \eqref{eq:rsh} constant, they could also be taken to be random. Two classical choices are either  to take   the vector $a$ to be uniformly distributed  on $S^{\kdn-1}$, or to take the $a_j$ as independent, standard Gaussians. In the former case we refer to random spherical harmonics under the spherical ensemble, while in the latter we refer to the Gaussian ensemble. The two are clearly related, due to the fact that, if the $a_j$ are Gaussian, then normalizing $a\to a/\|a\|$ gives a uniform variable on $S^{\kdn-1}$. Thus, the spherical harmonics under the spherical ensemble are a conditioned version of those under a Gaussian ensemble, with a corresponding statement going in the opposite direction.

The relationship between the two ensembles has been a recurrent theme in the general theory of Gaussian processes with a finite Karhunen-Lo\'eve expansion; i.e.\ processes which have a finite expansion similar to \eqref{eq:rsh}, although both the $\phi_j$  and the space over which they, and the process, are defined might be quite general (e.g.\  \cite{AT,S1,TK}). We shall give some more details a moment, but for the moment note that proofs based on this relationship typically only work when  the expansion is finite. If the processes in question have an infinite expansion, then approximating them with a finite expansion and  taking a passage to a limit has, to the best of our knowledge, only worked in situations in which the limit process is very smooth, typically at least $C^2$. 

Smoothness of  random spherical harmonics as $n\to\infty$  is most definitely not one of their properties, since the $n=\infty$ limit is not only not $C^2$,  but rather is a generalized function  (cf.\ \cite{CJW}). Consequently, 
one would not expect the passage to the limit mentioned in the previous paragraph to be at all relevant for them. The rather surprising result of this paper is that this is not exactly the case, and, with the right normalizations, connections between the spherical ensemble and integral geometry  which hold for the finite $n$ case still make sense as  $n\to \infty$.  In particular, we shall obtain 
explicit  formulae for the tail probability of the supremum of random spherical harmonics above  high levels,   and for  the expected Euler characteristic of the excursion sets (cf.\  \eqref{def1}). The derivations will rely on a very surprising result about a certain immersion of $S^d$ into the sphere $S^{\kdn-1}$, which has independent interest, and is really the main result of the paper. Thus we describe it first, then describe its implications for random spherical harmonics, and then close the introduction with a roadmap to the remainder of the paper.

\subsection{Spherical harmonics and the immersion}
The main result of the paper is actually a deterministic one, and rather simple to state.

%\vskip1.0truein
%be equipped with the round metric $g_{S^d}$.    
%Let $\Delta_{g_{S^d}}$ be the Laplace operator with respect to $g_{S^d}$ and $V_{g_{S^d}}$ the associated volume form (i.e.\ surface area).
%The   spherical harmonics  $\phi^{n,d}_j$ of level $n$ are  the eigenfunctions of   $\Delta_{g_{S^d}}$ for which
%\beq
%\label{eq:eigen}
%\Delta_{g_{S^d}} \phi^{n,d}_j(x) =-n(n+d-1)\phi_j^{n,d}(x).
%\eeq
%
%
%
%\subsection{Spherical harmonics}
%Let the unit sphere $S^d$ be equipped with the round metric $g_{S^d}$.    
%Let $\Delta_{g_{S^d}}$ be the Laplace operator with respect to $g_{S^d}$ and $V_{g_{S^d}}$ the associated volume form (i.e.\ surface area).
%The   spherical harmonics  $\phi^{n,d}_j$ of level $n$ are  the eigenfunctionss of   $\Delta_{g_{S^d}}$ for which
%\beq
%\label{eq:eigen}
%\Delta_{g_{S^d}} \phi^{n,d}_j(x) =-n(n+d-1)\phi_j^{n,d}(x).
%\eeq
%We normalize the eigenfunctions so that the $L^2$ norm of $\phi_j^{n,d}$ is $1$. 
% We denote the eigenspace $\mathcal H^d_n$ as the span of these  eigenfunctions, and it is classical that the dimension of $\mathcal H^d_n$ is \begin{equation}\label{dimn}k^d_n=\frac{2n+d-1}{n+d-1}{n+d-1\choose d-1} .\end{equation}

Consider the map  $i_n^d: S^d\to \mathbb R^{k^d_n}$, defined by  
\begin{equation}
\label{id}
i_n^d(x) \  = \ \sqrt{\frac{s_d }{k^d_n}}(\phi_{1}^{n,d},\cdots, \phi_{k^d_n}^{n,d}),
\end{equation}
where \beqq
s_d  \ = \  
%V_{g_{S^d}}(S^d)\ =    
 \frac{2\pi^{(d+1)/2}}{\Gamma ((d+1)/2)},
\eeqq
is the Euclidean surface area of $S^d$.   %, or, equivalently, its volume under the volume form $V_{g_{S^d}}$.

It is an easy calculation, that we shall carry out  in Section \ref{eeee}, that  $\|i_n^d(x)\| =1$ for all $x\in S^d$, so that $i_n^d$ is actually an mapping of spheres into spheres, viz. 
\begin{equation}\label{id2}i_n^d:\,\,\, S^d\to  S^{k_n^d-1}.\end{equation} 

As proved in \cite{N, Z}, this map is actually an immersion for sufficiently large $n$. Indeed, if $n$ is odd it is an embedding, while if $n$ is even then $i_n^d(S^d)\cong \mathbb RP^d$, the real projective space of dimension $d$. Furthermore, the pullback of the Euclidean metric to $S^d$ has the  leading order expansion 
\begin{equation}(i_n^d)^*(g_E)\cong c_dn^2 g_{S^d},
\end{equation}
where $c_d$ is a constant depending on $d$ and $g_{S^d}$ is the standard round metric on $S^d$.

Hence, roughly speaking, a geodesic of  unit length on  the unit sphere $S^d$ will be stretched  by a factor of order $n$ under the map $i_n^d$, and so  it is reasonable to expects that its image, as with that of the entire sphere,  becomes highly `twisted'   as $n$ grows.  An informative measure of twistedness is provided by the notion of {\it reach} or {\it critical radius}, which we shall define and  describe in Section \ref{radii}, and which is a measure of both the local and global smoothness of a set.  In general, the smaller the reach of a set, the less well behaved it is.
In view of the last three sentences, the following result, 
which shows that   there exists a uniform lower for the critical radii of the immersions as $n\to\infty$, is thus somewhat surprising:

\begin{theorem}\label{1} For  sufficiently large  $n$, the reach of the immersion $i_n(S^d)$ in $\mathbb R^{k_n^d}$ has a strictly positive, uniform in $n$, lower bound which  depends only on $d$.
\end{theorem}

%We shall prove this theorem for $S^2$ in Section \ref{2dproof} and for the general case in Section \ref{eeee}.
An explicit lower bound for the two dimensional case is given in \eqref{eq:lb2}, and for general in \eqref{eq:lbd}.
From Theorem \ref{1} it follows that there is a lower bound for the critical radius of   $i_n^d(S^d)$ considered as a subset of $S^{k_n^d-1}$. Let $\rho_d$ denote this new lower bound throughout the article.

With the deterministic Theorem \ref{1} in place, we can turn to its implications in a random setting.

\subsection{Random spherical harmonics}

As we have already mentioned, our results touch on random spherical harmonics under both the spherical and Gaussian ensembles. 
Both of these are objects of active research, much of the motivation coming from 
%By a random spherical harmonic we mean a sum of the form 
%\beq
%\label{eq:rsh}
%\Phi^d_n(x)=\sum_{j=1}^{k^d_n} a_j\phi_j^{n,d} ,
%\eeq
%where the $a_j$ are random variables. In view of \eqref{eq:eigen} it is immediate that, regardless of the properties of the $a_j$, each $\Phi^d_n$ is also a (random) eigenfunction of $\Delta_{g_{S^d}}$.
%
%There are two main choices for the $a_j$ in the literature. The more common is to take the $a_j$ as independent, identically distributed, Gaussian variables, in which case the $\Pdn$ are generally known as Gaussian spherical harmonics. More generally when the parameter space is a more general manifold and the spherical harmonics are replaced by the eigenfunctions of the corresponding Laplacian, they  are known  monochromatic random waves.  There is a significant and growing literature on Gaussian random spherical harmonics,
%much of it motivated by 
 Berry's conjectures in the 1970's (e.g.\ \cite{Berry}) linking them to the  eigenstates of  semi-classical, quantum, Hamiltonian systems, but more recently  motivated by intrinsic mathematical interest. Thus, for example, there is a large and growing mathematical literature on the nodal domains of these systems (e.g.\ \cite{NS,SW}), although its roots too are in the quantum mechanical applications. There is also a rich literature on exceedence probabilities (e.g.\ \cite{CX, VM, VM2}), while part of the general exceedence theory for Gaussian random fields (for which \cite{AT}) will be our basic reference) is actually motivated by the statistical analysis of the cosmic microwave background radiation data. 
 
 Throughout this paper,  we shall concentrate primarily on the spherical rather than the Gaussian ensemble. The reason is three-fold. Firstly, the calculations on reach in Sections \ref{2dproof} and \ref{eeee} are independent of the ensemble. Secondly,  when applying these 
 results one typically first treats the spherical ensemble, and then moves to the Gaussian ensemble via the conditioning argument described above. This is standard, and so 
 we shall not treat it further. Finally, under the spherical ensemble, random spherical harmonics also have a property that makes them of intrinsic mathematical interest. It follows from the properties of (deterministic) spherical harmonics that, in the spherical case,  
\beq
\label{eq:pL2}
\|\Phi_n^d\|_{L^2}\ =\  \int_{S^d}  \big| \Pdn(x)\big|^2 \,dV_{g_{S^d}}\ = \  1,
\eeq
where we write $V_{g_{S^d}}$ for volume measure with respect to $g_{S^d}$. Put more simply, $V_{g_{S^d}}$  measures surface area on 
$S^d$, so that, for example, $s_d= V_{g_{S^d}}(S^d)$. Note that, while the $\Pdn$ are random, the equality in \eqref{eq:pL2} holds for each realization, or, to be more precise, with probability one.

%{\color{red} Comment: The `probability one' is  needed here. For example, most events related to the uniform distribution governing the coefficients are only defined up to probability one.}

From this it follows, if we now write $\Hdn$ to denote the $n$-th eigenspace of $\Delta_d$ generated by the solutions of the wave equation \eqref{eq:eigen}, and 
 $S\Hdn$ to denote $L^2$-sphere in this space,  that $\Pdn$, under the spherical ensemble,  is a random element of $S\Hdn$. Thus it provides a mathematical model for studying this space.  
 
 Two results that at first seem somewhat at odds with \eqref{eq:pL2}  are due to Burq and Lebeau  \cite{BL}.  To state them we need some notation. In particular, we shall denote probabilities and expectations under the spherical ensemble by $\Pmnd$ and $\Emnd$, respectively. 
 Then Burq and Lebeau showed that, for $u \geq 1$,
 and all $\alpha < s_d$, 
\beq
\label{eq:prob}
\Pmnd\left\{  \sup_{S^d}\left|\Phi_n^d(x)\right| \ > \ u\right\}
\ \leq \ C n^{-d(1+d/2)}e^{-\alpha u^2}.
\eeq
%  where we use  $\Pmnd$ to denote the probability measure induced on the Borel subsets of $S\Hdn$ by the uniform spherical harmonics of degree $n$. 
The result \eqref{eq:prob}  is typical of what we referred to above as an exceedence probability. The second, related, result established the 
logarithmic growth of the expectation of  suprema; viz.\   for some $0<c<C<\infty$,
\beq
\label{eq:BL}
c\sqrt{\log n}\ \leq\   \Emnd  \left\{\sup_{S^d}\left|\Phi_n^d(x)\right|\right\} \ \leq   \ C\sqrt{\log n}.
\eeq

Combining \eqref{eq:pL2}--\eqref{eq:BL} we obtain a picture of sample paths for $\Pdn$ which, while  almost surely $L^2$-integrable, have local behavior which grows increasingly erratic as $n\to\infty$, with the the supremum having an exponential concentration of measure around $\sqrt {\log n}$. 

There are also analogues of \eqref{eq:prob} and \eqref{eq:BL} under the Gaussian ensemble. The close connection between the above results for the two ensembles is not coincidental, but, rather, is related to the fact that the spherical ensemble is a conditional version of the Gaussian ensemble as we mentioned above.

However, it turns out that, despite the irregular behaviour of random spherical harmonics for large $n$,  the uniform lower bound that Theorem \ref{1} provides for the critical radii of the immersions $i_n^d$ actually allows one to exploit this general approach to prove a number of interesting results.

\subsection{Consequences for random spherical harmonics}

%\cite{AET} somewhere

%
%However, we shall concentrate on a different family of processes for which far  fewer results are available, and in which 
%$a=(a_1,\dots,a_{\kdn})$ is chosen with respect to the uniform measure, or uniform spherical ensemble, $\mdn$ on $S^{\kdn-1}$. Thus, trivially,  
% \beq
% \label{eq:asums}
%\|a\|\ = \  \sum_{j=1}^{k^d_n} |a_j|^2  \ =\  1,
%\eeq

%%\beq
%%\label{eq:EL2}
%%\E\left\{ \left| \Pdn(x)\right|^2 \right\}  \ = \ \kdn / s_d.
%%\eeq
%%for all $x$, in which case
%\section{Critical radius}
%\label{sec:critical}
%The study of the probability and the Euler characteristic of the excursion set of the random fields is one of the main topic in random geometry. In this paper, we consider such quantities for the normalized random spherical harmonics $(S\mathcal H^d_n, d\mdn)$ as $n$ large enough, here, $$S\mathcal H^d_n=\left\{\Phi_n^d\in \mathcal H_n^d:\, \|\Phi_n^d\|_{L^2}=1 \right\}$$ and $d\mdn$ is the spherical ensemble.
%
% %Alternatively, we consider the random spherical harmonics in $S\mathcal H^d_n$ in the form of 
%%$$
%%\Phi^d_n(x)=\sum_{j=1}^{k^d_n} a_j\phi_j^{n,d} ,\,\,\,\, \sum_{j=1}^{k^d_n} |a_j|^2=1.
%%$$
%% {\color{red} Need reference here!}. 

%, and 
%their behaviour as $n\to\infty$ has raised considerable interest.
%{\color{red} We need some  references here about where older results appeared.} .....
We need some notation. For $u>0$, denote the  excursion sets of $\Pdn$ by
\beq
\label{def1}
\And (u) \ =  
%A_u\left(\Pdn,S^d\right) \ = 
\  \left\{x\in S^d:\,\,  \Phi^d_n(x)  > u  \right\}.
\eeq

%If, however, we turn to even higher levels, then it turns out that we can derive an exact formula for the excursion probability. Moreso, it turns out that this formula is equivalent to the mean Euler characteristic of  excursion sets, which we define by 
%\begin{defn} For $0\leq \rho\leq \rho_d$,  the upper tail excursion set is defined as 
 \begin{theorem}\label{34}
 Let $\Pdn$ be spherical harmonics under the spherical ensemble.  Then there exist constants $\rho_d>0$ such that, for sufficiently large $n$,  
 and for  all  $u > \sqrt {k^d_n/s_d}\cos\rho_d$,
   \beq
  \label{eq:main} 
  \mathbb P_{\mu_n^d} \left\{\sup_{S^d}  {\Phi^{d}_n(x)}> u \right\} 
   =  \kappa\,\E_{\mu_n^d}\left\{\chi\left(\And\left(u\right) \right)\right\},
 \eeq
where $\kappa=1/2$ if $n$ is even and $1$ if $n$ is odd, and $\chi (A)$ denotes the Euler characteristic of the set $A$. 
%
%When the probability and expectation in \eqref{eq:main} are taken under the spherical ensemble, we require $u>\sqrt {k^d_n/s_d}\cos\rho_d$.
%
%When they are taken under the Gaussian ensemble, we require $u>$ {\color{red} something which needs to be determined}

\end{theorem}

% \begin{theorem}\label{34}
% For $n$ large enough for Theorem \ref{1} to hold, and with $\rho_d$ as defined there, for the uniform spherical harmonic $\Pdn$ we have,
% for all $0<\rho<\rho_d$, 
%  \beq
%  \label{eq:main} 
%  \mathbb P_{\mu^d_n}\left\{\sup_{S^d}  \frac{|\Phi^{d}_n(x)|}{\sqrt {{k^d_n}/{s_d}}}>\cos\rho\right\} 
%   =  \kappa\,\E_{\mu^d_n}\left\{\chi\left(\And\left(\sqrt {k^d_n/s_d}\cos\rho\right) \right)\right\},
% \eeq
%where $\kappa=1/2$ if $n$ is even and $1$ if $n$ is odd. 
%\end{theorem

The factor of $\kappa$ here is due to the fact that $\Pdn(S^d)$ is isomorphic to $S^d$ for $n$ odd, it is isomorphic to 
$\real P^d$ for $n$ even.   This affects tube formulae, which are the key to the probability calculation leading to \eqref{eq:main}, but not   the Euler characteristic. %The $\rho_d$ here is not the same as that in Theorem \ref{1}, as will be clear from the proof. The $\rho_d$ in Theorem \ref{1} is related to the reach of $i_n^d(S^d)$ as a subset of $\real^{\kdn}$, while here it is related the its reach as a subset of $S^{\kdn-1}$.

Note that \eqref{eq:main} is an exact result (for quantifiably large $u$) and not an asymptotic equivalence as is more common, for example, in the Gaussian literature.

Precise expressions for the probability and expectation in Theorem \ref{34} are basically already available in the literature, and lead to the following set of results, in which $P_{n,d}$ denotes the $n$-th Legendre polynomial of order $d$. 

%\begin{proposition} 
%\label{prop:1} Under the conditions of Theorem \ref{34},
%\beq
%\label{eq:prob}
%\mathbb P_{\mu^d_n}\left(\sup_{S^d}  \frac{|\Phi^{d}_n(x)|}{\sqrt {{k^d_n}/{s_d}}}>\cos\rho\right)  = \frac {\kappa}{s_{{k_n^d-1}}}\sum_{j=0}^d f_{k_n^d,j}(\rho) [P_{n,d}'(1)]^{j/2}\mathcal L_j(S^d),
%\eeq

\begin{proposition} 
\label{prop:2} Under the conditions of Theorem \ref{34},
\beq
\label{eq:probc}
&& \mathbb P_{\mu^d_n}\left\{\sup_{x\in S^d} \Phi^{d}_n(x)>\ u \right\} \\ &&\qquad  = \frac {\kappa}{s_{{k_n^d-1}}}\sum_{j=0}^d f_{k_n^d,j}
\left( \cos^{-1}(u/\sqrt {{k^d_n}/{s_d}} )   \right) [P_{n,d}'(1)]^{j/2}\mathcal L_j(S^d),   \nonumber 
\eeq
%$k^d_n$ is  the dimension of $\mathcal H_n^d$ given in \eqref{dimn},  
where $\mathcal L_{j}(S^d)$ are the  $j$-th Lipschitz-Killing curvatures  of the unit sphere $S^d$,  given explicitly by 
\eqref{curvatures},
 %\eqref{geometry:LK-sphere:equation},
   and the  $f_{k_n^d, j}$ are  functions defined by \eqref{funtions} below. 
% where
%$$C=2 \,\,\, \mbox{for odd}\,\,\,n; \,\,\, C=1 \,\,\, \mbox{for even}\,\,\,n.$$
\end{proposition}
As a direct corollary of Theorem \ref{1}  and   Proposition \ref{prop:2}  we have the following result for $S^2$: 
\begin{corollary}\label{tjm2}
%Let's choose the uniform bound of the sphere $S^2$ in Theorem \ref{1} as $\rho_2$, then we have the explicit formula
%\begin{align*}\mathbb E_{\mu_n}^{\chi}\left\{x\in S^2: \frac{|\Phi_n(x)|}{\sqrt {\frac{2n+1}{4\pi}}}>{\cos\rho}\right\}=\frac{C\Gamma(n+\frac12)}{\Gamma(n-1)\pi^{\frac 32}} \int_0^\rho\sin ^{2n-3}(r)\left\{ (n^2+n)\pi+\frac{\pi\sin^2 (r)}{n-1}\right\} dr\end{align*}
For $u>\sqrt{(2n+1)/4\pi} \cos (\rho_2)$,
%\beqq
%&&\mathbb P_{\mu_n}\left\{\sup_{S^2} \frac{|\Phi_n(x)|}{\sqrt {{2n+1}/{4\pi}}}>{\cos\rho}\right\}
% \\  &&\qquad\qquad =
%\kappa\,\mathbb E_{\mu_n}\left\{ \chi\left(A_n^2\left(\sqrt{(2n+1)/4\pi}\cos\rho\right)\right)\right\}
%\\ && \qquad\qquad = \frac{\kappa\Gamma (n+\frac{1}{2})}{\pi^{1/2}\Gamma(n-1)} \int_0^\rho\sin ^{2n-3}(r)
%\left\{2(n^2+n)  \right.
%\\  &&\qquad\qquad\qquad\qquad
%\ \times \left. \left(1-\frac{2n-1}{2n-2}\sin^2(r)\right)+\frac{2\sin^2 (r)}{n-1}\right\} dr.
%\eeqq
\beqq
&&\mathbb P_{\mu_n^2}\left\{\sup_{S^2} \Phi_n^2(x)>  u \right\}
 \\  &&\qquad =
%\kappa\,\mathbb E_{\mu_n}\left\{ \chi\left(A_n^2\left(\sqrt{(2n+1)/4\pi}\cos\rho\right)\right)\right\}
%\\ && \qquad\qquad 
 \frac{\kappa\Gamma (n+\frac{1}{2})}{\pi^{1/2}\Gamma(n-1)} \int_0^{  \cos^{-1}(u/\sqrt{(2n+1)/4\pi})  }\sin ^{2n-3}(r)
\\  &&\qquad\qquad
\ \times \Big\{2(n^2+n)   \left(1-\frac{2n-1}{2n-2}\sin^2(r)\right)+\frac{2\sin^2 (r)}{n-1}\Big\} \,dr.
\eeqq
% where
%$$C=2 \,\,\, \mbox{for odd}\,\,\,n; \,\,\, C=1 \,\,\, \mbox{for even}\,\,\,n.$$
\end{corollary}

%The corresponding result for the Gaussian ensemble is a little tidier to write out, and is 
%
%\begin{proposition} 
%\label{prop:1} Under the conditions of Theorem \ref{34},
%\beq
%\label{eq:probg}
%&& \mathbb P_{\gamma^d_n}\left\{\sup_{x\in S^d} \Phi^{d}_n(x)>\ u \right\}
%\ = \ \text{\color{red} Whatever it is}
%\eeq
%\end{proposition}
%
%

%
%In this article, we study the probability of the supremum above some high level for the normalized random spherical harmonics instead of the probability around the mean. 
%
%There are several classical methods to study the probability of the supremum of random functions above high levels, one is the tube method developed by Sun \cite{S1} and another one is the Euler characteristic method \cite{AT}, these two method are equivalent \cite{TK}. 
%
%Theorem \ref{1} implies that there is also a uniform lower bound for the critical radius of $i_n^d(S^d)$ in the ambient space $S^{k_n^d-1}$ if we consider the image $i_n^d(S^d)$ as a submanifold in   $S^{k_n^d-1}$ (see \eqref{id2}). Once we have the lower bound of the critical radius,  we can choose some smaller $\rho_d$ and we have 
%
%Regarding the excursion set of upper tail,  we have, 

% \begin{thm}\label{222} We have the following formula for  Lipschitz-Killing curvatures, 
 %$$\mathbb{E}_{\mu_n^d}\mathcal L_i \left\{x\in S^d: \frac{|\Phi^d_n(x)|}{\sqrt {\frac{k^d_n}{|S^d|}}}>\cos\rho \right\}=\frac {1}{|S^{k_n^d-1}|}\sum_{j=0}^{d-i} \frac{s_{i+1}s_{j+1}}{2s_{i+j+1}}
 %[P_{n,d}'(1)]^{j/2}\mathcal L_{i+j}(S^d)f_{k_n^d,j}(\rho),$$
 %where $s_i$ is the area of the unit sphere $S^{i-1}$. 
%\end{thm}

The simple structure of the two-dimensional result in Corollary \ref{tjm2} makes it easy to understand the large deviation nature of the result. In particular, since  $\Phi_n^2(x) = \sum  a_j \phi^{n,2}_j(x)$, it follows that
\beqq
 \left|\Phi_n^2(x)\right|^2 \ \leq \ (\sum_{j=1}^{k_n^2}  a_j^2)(  \sum_{j=1}^{k_n^2} \left|\phi_j^{n,2}(x)\right|^2)
 \ =\  \sum_{j=1}^{k_n^2} \left|\phi_j^{n,2}(x)\right|^2 \  =\ \frac{2n+1}{4\pi},
\eeqq
the last equality coming from \eqref{eq:K} and \eqref{eq:K1} below. Thus Corollary  \ref{tjm2} relates only to the range
$u\in [ \sqrt{(2n+1)/4\pi} \cos (\rho_2), \sqrt{(2n+1)/4\pi}]$, which makes it a large deviation result. As opposed to most large deviation results, however, this one is quite unique in the fact that the exceedence probability is precise, and not just an approximation. 

Obviously, a similar comment holds for Theorem \ref{34} and Proposition \ref{prop:2} for general $d$ and $n$.

\subsection{A roadmap}
We now turn to proving these results. In the following section we collect some results on spherical harmonics, and in Section \ref{radii} we do the same for critical radii. Section \ref{2dproof} then proves Theorem \ref{1} for the case $d=2$, while Section \ref{eeee} treats the case of general $d$. Section \ref{ECTF} proves the remaining results, and in the final Section \ref{CLOSING} we collect some comments relating our results to others in the literature and mention some interesting open problems.

% 
%\section{SPARE STUFF}
%
%
%
%We start with   normalized, {\it deterministic,} spherical harmonics of level $n$ on the unit sphere $S^d$, and use them to map 
%-- via the immersion $i_n^d$ of \eqref{id}  -- 
%the sphere into a high-dimensional Euclidean space. 
%%as $n$ tends to infinity. 
%The main, and a priori unexpected, result is that there is a uniform lower bound to the critical radius for the immersion  as $n\to\infty$. Two results follow from this fact, but for certain {\it random} spherical harmonics of high order under the uniform spherical ensemble. In particular, we obtain 
%explicit  formulae for the tail probability of the supremum of random spherical harmonics above  high levels,   and for  the expected Euler characteristic of the excursion sets (cf.\  \eqref{def1}). 
%%Then we  can get all the Lipschitz-killing curvatures of the excursion set by the Kinematic Fundamental Formulas (KFF). 
%
%To explain all this we need some notation, which we shall follow with some background and motivation.

\section{Spherical harmonics on $S^2$}\label{sec1}
In this section we shall collect a number of results specific to spherical harmonics on $S^2$, which we shall use in our proof of Theorem \ref{1}. Similar results hold in higher dimensions, but, for the moment, we stay in dimension 2. We then look at immersions.

Since, for this and most of the following two sections, we shall be dealing with the case of $S^2$, we shall drop the  the superscript 2 whenever it does not lead to ambiguities. Thus, $i^2_n$ becomes $i_n$, $\Phi_n^2$ becomes $\Phi_n$,  $P_{n,2}$ becomes $P_n$, and so forth.

\subsection{Some basic facts}
Consider the unit sphere $S^2$ equipped with the round metric $g_{S^2}$ and  with associated Laplacian $\Delta$.
The   spherical harmonics  $\phi^n_j$  are  then the  eigenfunctions of 
$$\Delta  %{g_{S^2}}
 \phi^n_j(x) =-n(n+1)\phi_j^n(x).$$
We normalize the eigenfunctions so that the $L^2$ norm of $\phi_j^n$ is $1$, and 
denote by  $\mathcal H_n$  their span. The  dimension of $\mathcal H_n$ is $2n+1$. Since the Laplacian  is invariant under  rotation, 
 $\mathcal H_n$
is invariant under the action $\phi(x) \to \phi(Qx)$ for $Q \in
SO(3)$. Moreover, if $\{\phi^n_j(x)\}$ is an orthonormal basis of $\mathcal H_n$, so is
  $\{\phi_j^n (Qx)\}$. 
  
 Let $\mathcal H_n$  be spanned by $\{\phi_{-n}^n,\dots, \phi_0^n,\dots, \phi_{n}^n\}$. We denote $K_n$ as  the spectral projection from the $L^2$-integrable functions to the spherical harmonics of level $n$, so that 
 $$K_n:\,\,\, L^2(S^2)\to \mathcal H_n(S^2).$$ Then the kernel of $K_n$ is given by
\beq
\label{eq:K}
K_n(x,y)= \sum_{j=-n}^n \phi_j^n(x)\phi_j^n(y).
\eeq
In fact, the spectral projection kernel has the following explicit formula \cite{AH,So}.
$$K_n(x,y)=  \frac{2n+1}{4\pi}P_n(\cos\Theta(x,y)),$$
 where %$P_n(x)$ is the Legendre polynomial or order $n$ and 
 $\Theta(x,y)$ is the angle between the vectors $x, y\in S^2$.  The Legendre polynomials (of order 2) are defined by
 $$P_n(x)=\frac{1}{2^n n!} \frac{d^n(x^2-1)^n}{dx^n}.
 $$
 Some basic facts that we shall require are \cite{AH,T} 
 \beq
 \label{eq:basic}
  P_n(1) = 1;\quad P_n'(1)=\frac{n^2+n}2;\quad  -1\leq P_n(x)\leq 1, \ \  \mbox{for}\,\,x\in[-1,1],
 \eeq
 and 
 \beq
 \label{symmetry}
 P_n(-x)\ =\ (-1)^nP_n(x).
 \eeq
Thus, on  the diagonal,  the kernel satisfies
\beq
\label{eq:K1}
K_n(x,x)\ =\   \frac{2n+1}{4\pi}P_n(1)\ =\ \frac{2n+1}{4\pi}.
\eeq

\subsection{Immersions}
Consider the map
\begin{equation}\label{randomem}i_n:\, S^2\to \mathbb R^{2n+1},\quad x\to \sqrt{\frac{4\pi}{2n+1}}\left(\phi^n_{-n}(x), \dots, \phi^n_0(x),\dots, \phi^n_n(x)\right).
\end{equation}
For large enough $n$, this map is an immersion \cite{Z}. 

Defining  the normalized kernel 
$$\Pi_n(x,y)\ =\  \frac {4\pi}{2n+1}K_{n}(x,y)\ =\ P_n(\cos\Theta(x,y))$$
we have that the norm of $i_n(x)$ is given by
 $$\|i_n(x)\|^2\ =\ \frac{4\pi}{2n+1}\sum_{j=-n}^n|\phi_{j}(x)|^2\ =\ \Pi_n(x,x)\ =\ 1.
 $$
 Thus $i_n$ is actually a map from $S^2$ to $S^{2n-1}$,  and  the pullback of the Euclidean metric is  
\begin{equation}\label{metric}g_n=i_n^*(g_E) =\frac{n^2+n}2 g_{S^2},\end{equation}
where we use $g_E$ to denote  the standard Euclidean metric. 
While this fact is well known (cf.\ \cite{N, Z}) it will follow, en passant, from calculations below (cf.\ the argument surrounding \eqref{eq:foreqmetric}).

The distance between two points of the immersion is given by 
\beq
\label{dis}
\|i_n(x)-i_n(y)\|^2 &=& \Pi_n(x,x)+\Pi_n(y,y)-2\Pi_n(x,y)\\ &=&2(1-P_n(\cos\Theta(x,y))),  \nonumber
\eeq
 and so it follows from  \eqref{symmetry} that $i_n$ is an embedding for $n$ odd but   identifies antipodal
points   for $n$ even. Thus, in the  case of even $n$, it follows that $i_n(S^2)\cong \mathbb RP^2$. 

%Furthermore, and in fact this map is embedding as $n$ large enough. 

\section{The critical radius of $i_n(S^2)$}\label{radii}

The modern notion of  reach, or critical radius (terms which we shall use interchangeably)  seems to have appeared first in the classic paper \cite{FED}  of Federer, in which he introduced the notion of sets with positive reach and their associated curvatures and curvature measures. In doing so, Federer was able to include, in a single framework, Steiner's tube formula for convex sets and Weyl's tube formula for $C^2$ smooth submanifolds of $\real^n$.
The importance of this framework extended, however, far beyond tube formulae, as it became clear that much of the theory surrounding convex sets could be extended to sets that were, in some sense, locally convex, and that  the reach of a set was precisely the way to quantify this property.

To be just a little more precise, suppose $N$ is a smooth manifold  embedded in an ambient manifold $\wN$. Then
the local reach at a point $x\in N$ is the furthest distance one can travel, along any geodesic in $\wN$  based at $x$ but normal to $N$ in $\wN$, without meeting a similar vector originating at another point in $N$. The (global) reach of $N$ is then the infimum of all local reaches. As such it is related to local properties of $N$ through its second fundamental form, but also to global structure, since points on $N$ that are far apart in a geodesic sense might be quite close in the metric of the ambient space $\wN$. 

There are many, equivalent, formal definitions, of reach, but we shall take as our definition a result which is actually a theorem  of Takemura and Kuriki \cite{TK}, that states that 
for a compact Riemannian manifold $N\subset \mathbb R^k$,   
 the critical radius is given by 
\begin{equation}
\label{def:reach}
r_c(N) =\inf_{x,y\in N} \frac{\|x-y\|^2}{2 \|P^{\perp}_y(x-y)\|},\end{equation}
where $P^{\perp}_y(x-y)$ is the projection of  $x-y$
 to the normal bundle at $y$. 
 
 This is actually all we need for the remainder of the paper, and so for the reader interested to know more about critical radii we  refer you to the review \cite{Thale}  for an excellent coverage of the history and uses of this notion in Mathematics as a whole,  and to the expository sections of \cite{AKTW} to see why it is an important property in the theory of  random processes.

Our interest  now, however, is  in the critical radii of the immersions  $i_n(S^2)$ in $\mathbb R^{2n+1}$, and so we now concentrate solely on this. 

Rewriting \eqref{def:reach} for this setting, we have that  the critical radius of $i_n(S^2)$ is given by
\begin{equation}
\label{critical}
r_{c,n}\ :=\ \inf_{x,y\in S^2} \frac{\|i_n(x)-i_n(y)\|^2}{2 \|P^{\perp}_{i_n(y)}(i_n(x)-i_n(y))\|}.
\end{equation}
The  numerator here is given by \eqref{dis}, and the first step  regarding  the denominator is to compute  the projection of the vector $i_n(x)-i_n(y)$ to the normal space,
i.e.\ the orthogonal complement, in $\real^{2n+1}$,  of the tangent space $T_{i_n(y)} i_n(S^2)$.  % at $i_n(y)$. 

To this end, we move to  polar coordinates 
$$x\ =\ (\sin\theta_x\sin \phi_x,\,  \sin\theta_x\cos\phi_x,\, \cos\theta_x),$$
 with $0\leq \theta\leq \pi$ and $0\leq \phi<2\pi$, and the similar definition for $y$. 
 
 Note that the normalized projection kernel $\Pi_n$ is constant on diagonal, and so
 \begin{equation}\label{ddddddd}\d_\theta\Pi_n(y,y)=\d_\phi\Pi_n(y,y)=0.
 \end{equation}

We rewrite the normalized kernel in  polar coordinates as
\beq
\label{d}
\Pi_n(x,y)  &=& P_n(\sin\theta_x\sin\phi_x\sin\theta_y\sin\phi_y+ \\ &&\qquad\qquad \sin\theta_x\cos\phi_x\sin\theta_y\cos\phi_y+\cos\theta_x\cos\theta_y).  \nonumber 
\eeq
This yields
\beqq
\d_{\theta_y}\Pi_n(x,y)  &=& P'_n(\cos\Theta(x,y))
\big[\sin\theta_x\sin\phi_x\cos\theta_y\sin\phi_y  \\ &&\qquad\qquad +\sin\theta_x\cos\phi_x\cos\theta_y\cos\phi_y-\cos\theta_x\sin\theta_y\big],
\eeqq
and  
\beqq 
\d_{\phi_y}\Pi_n(x,y)&=&P'_n(\cos\Theta(x,y)) \big[\sin\theta_x\sin\phi_x\sin\theta_y\cos\phi_y
\\ && \qquad\qquad -\sin\theta_x\cos\phi_x\sin\theta_y\sin\phi_y\big].
\eeqq
Further differentiation now  yields   
\beq
\label{e1}
\d_{\theta_x}\d _{\theta_y}\Pi_n(x,y)|_{x=y}&=&  P'_n(1),
%Since $P_n'(1)=P'_{n-1}(1)+nP_n(1)$, thus  $P_n'(1)=\frac {n^2}2+\frac n 2$.
\\  \label{e2}\d_{\phi_x}\d _{\phi_y}\Pi_n(x,y)|_{x=y}&=& P'_n(1)\sin^2\theta,\\
 \label{e3}\d_{\theta_x}\d _{\phi_y}\Pi_n(x,y)|_{x=y}&=& 0.
 \eeq
An easy consequence of these three identities  is the fact, given in \eqref{metric}, that the pullback, under $i_n$, of the Euclidean metric on $\mathbb R^{2n+1}$ is a scaled  version of the standard metric on $S^2$. To see this, note that the pullback is just 
\beq
\label{eq:foreqmetric}
\sum d\phi^n_j(x)\otimes d\phi^n_j(x),
\eeq
which we can write as $d_xd_y\Pi_n(x,y)|_{x=y}$. Since  the differential operator $d $ is global, it is unchanged if we take derivatives with respect to the angle variables $\theta$ and $\phi$.
%which is 
%\beq
%\label{eq:foreqmetric}
%\left( \d_\theta\d_\theta\ +\, 
%\d_\theta\d_\phi \, +\, 
%\d_\phi\d_\phi  \right)\Pi_n.
%\eeq
Applying now \eqref{e1}--\eqref{e3} and \eqref{eq:basic} immediately establishes \eqref{metric}.

 With polar notation, it is easy to see that the  tangent subspace at $i_n(y)$ is spanned by the vector $\left\{ \frac{\d i_n}{\d\theta}(y), \frac{\d i_n}{\d\phi}(y)\right\}$.   \eqref{e3} implies that these two vectors are orthogonal, i.e.,  $$\langle  \frac{\d i_n}{\d \theta}(y), \frac{\d i_n}{\d \phi}(y)\rangle=0.$$
 
  Thus the projection of  $i_n(x)-i_n(y)$ to the tangent space is 
$$
p_x(y)\ :=\ \frac{\langle i_n(x)-i_n(y), \frac{\d i_n}{\d \theta}(y) \rangle }{  |\frac{\d i_n}{\d \theta}(y)|^2}   \frac{\d i_n}{\d \theta}(y) +\frac{\langle i_n(x)-i_n(y), \frac{\d i_n}{\d \phi}(y) \rangle }{  |\frac{i_n(y)}{\d \phi}|^2}   \frac{\d i_n}{\d \phi}(y),
 $$
which can be rewritten as 
\beqq
%p_x(y) = 
 \frac{\d_{\theta_y}\Pi_n(x,y)- \d_{\theta_y}\Pi_n(y,y)}{ \d_{\theta_x}\d_{\theta_y}\Pi_n(x,y)|_{x=y}}   \frac{\d i_n}{\d \theta}(y) + \frac{\d_{\phi_y}\Pi_n(x,y)- \d_{\phi_y}\Pi_n(y,y)}{ \d_{\phi_x}\d_{\phi_y}\Pi_n(x,y)|_{x=y}}   \frac{\d i_n}{\d \phi}(y).
\eeqq
Applying  \eqref{ddddddd} to the above gives
\begin{equation}
\label{ra}
p_x(y)\ =\   \frac{\d_{\theta_y}\Pi_n(x,y) }{ \d_{\theta_x}\d_{\theta_y}\Pi_n(x,y)|_{x=y}}   \frac{\d i_n}{\d \theta}(y) + \frac{\d_{\phi_y}\Pi_n(x,y) }{ \d_{\phi_x}\d_{\phi_y}\Pi_n(x,y)|_{x=y}}   \frac{\d i_n}{\d \phi}(y).\end{equation}

%{\color{red} There is something wrong here. Which equation do you mean? It is definitely not \eqref{randomem}!} 

It follows that the squared norm  of the projection $p_x(y)$ in \eqref{ra} can be written as 
\begin{equation}\label{ppp}\|p_x(y)\|^2=\frac{|\d_{\theta_y}\Pi_n(x,y)|^2 }{ \d_{\theta_x}\d_{\theta_y}\Pi_n(x,y)|_{x=y}} +\frac{|\d_{\phi_y}\Pi_n(x,y)|^2 }{ \d_{\phi_x}\d_{\phi_y}\Pi_n(x,y)|_{x=y}} .   \end{equation}
Thus we can express the  the critical radius  \eqref{critical}  as
$$r_{c,n}\ =\ \inf_{x,y\in S^2} \frac{\|i_n(x)-i_n(y)\|^2}{2\sqrt{\|i_n(x)-i_n(y)\|^2-\|p_x(y)\|^2}}.$$
By rotation invariance, it is clear that each of the terms within the infimum here are dependent only on the relative positions of $x$ and $y$, and so the local radius is actually the same at the image of each point on the sphere. Thus,  it suffices  to consider the local critical radius at  any point. Choosing $x=(0,0,1)$ for this point, the critical radius $r_{c,n}$ can be written as
\begin{equation}
\label{higherg}r_{c,n} \ =\ \inf_{y\in S^2 }\frac{\|i_n((0,0,1))-i_n(y)\|^2}{2\sqrt{\|i_n((0,0,1))-i_n(y)\|^2-\|p_{(0,0,1)}(y)\|^2}}.\end{equation}
For $x=(0,0,1)$, we can write the coordinates of $y$ as $(\theta, \phi)$, and so  \eqref{d}--\eqref{e3} become  

 \beqq
\Pi_n((0,0,1),y)  &=& P_n(\cos\theta), \\
 \d_{\theta_x}\d _{\theta_y}\Pi_n(x,y)|_{x=y}&=&  P'_n(1),
\\ \d_{\phi_x}\d _{\phi_y}\Pi_n(x,y)|_{x=y}&=& P'_n(1)\sin^2\theta,
\\ \d_{\theta_y}\Pi_n(x,y)|_{x=(0,0,1)}&=&- P'_n(\cos\theta)\sin\theta,
 \\ \d_{\phi_y}\Pi_n(x,y)|_{x=(0,0,1)}&=&0.
 \eeqq
Similarly,  \eqref{dis} becomes
$$ {\|i_n((0,0,1))-i_n(y)\|^2}\ =\ 2(1-P_n(\cos\theta)),$$
and \eqref{ppp} reads
$$\|p_{(0,0,1)}(y)\|^2=\frac{[ P'_n(\cos\theta)\sin\theta]^2}{P_n'(1)}.$$
%$$\|P^{\perp}_{i_n(x)}(i_n(x)-i_n(y))\|^2=2-2P_n(\cos\theta)-[\frac{ P'_n(\cos\theta)\sin\theta}{P_n'(1)}]^2P_n'(1)$$
%we have to find 
Hence, we can now finally rewrite the critical radius of $i_n(S^2)$ in $\real^{2n+1}$ as
\beq
\label{eq:rcn}
r_{c,n} \ =\  \inf_{\theta\in [0,\pi]} \frac{1-P_n(\cos \theta)}{\sqrt{2-2P_n(\cos\theta)-\frac{ [P'_n(\cos\theta)\sin\theta]^2}{P_n'(1) }}},
\eeq
and we are now in a position to begin the more serious steps in the proof of Theorem \ref{1}, at least  for the case $d=2$.

%let's take $\theta=\pi$, why we get $\frac1{\sqrt {2}}$ which is samller than $\sqrt{\frac 83}$? something is wrong?

%Let's assume n is odd, otherwise, it's not an embedding,  so that $P_n(-x)=(-1)^nP_n(x)=-P_n(x)$, this implies that there is a symmetry about 
%$\theta=\frac \pi 2$ for the above expression. 

\section{Proof of Theorem \ref{1} for $S^2$}\label{2dproof}

In view of the preceding section, in order to prove Theorem \ref{1} for the case $d=2$ we need to provide a lower bound for the expression given in 
\eqref{eq:rcn} that is independent of $n$, at least for $n$ large enough.

Note firstly that  $P_n(\cos\theta)$ is symmetric  (anti-symmetric) about $\theta= \pi /2$ for $n$ even (odd). Thus, for $n$ even,  it suffices to consider $\theta\in [0,\pi/2]$ in \eqref{eq:rcn}. For the moment we shall assume that $n$ is even, and then discuss the odd $n$ case at the end of the section. 

So, with $n$ even, fix a positive constant $c$ and divide $[0, \pi /2]$ into the three subintervals 
$$
[0,   c /n], \ \  [c/n, n^{- 3/4}], \ \ [n^{- 3/4},  \pi /2].
$$
 For the first two, short range, subintervals, our  strategy will be to study the rescaling limit of the projection kernel and its derivatives. The infimum for the third subinterval  will follow  directly from the rapid decay of the projection kernel and its derivatives.  The entire 
  proof is based on   Hilb's asymptotics for Legendre polynomials \cite{T}, specifically, there exists a (uniform in $n$) constant $c$, for which 
  \begin{equation}
P_n(\cos\theta )\ =\ 
\left(\frac\theta{\sin\theta}\right)^{1/2}J_0\left((n+\smallhalf)\theta\right)\,+\,R_n(\theta),
\end{equation}
where
 \begin{equation}
R_n(\theta)=
\begin{cases}
\theta^2 O(1), &0\leq \theta\leq   c/n,\\
 \theta^{1/2}O(n^{-  3/2}),  &c/n\leq \theta \leq   \pi /2,
\end{cases}
\end{equation}
and    $J_0(\theta) $ is the Bessel function of order 0.

The global  infimum is then 
\beq
\label{eq:3parts} 
\inf_{\theta\in [0,\pi/2]}  \ =\ \min\left\{\inf_{[0,   c/n]},\  \inf_{[   c/n,  n^{ -3/4}]}, \ \inf_{[ n^{ -3/4},   \pi /2]}\right\} =: \min \left\{I_n, II_n, III_n\right\}.
\eeq
Consider the first infimum here:
$$I_n=\inf_{\theta\in [0, c/n]}\frac{1-P_n(\cos \theta)}{\sqrt{2-2P_n(\cos\theta)-\frac{ [P'_n(\cos\theta)\sin\theta]^2}{P_n'(1) }}}. $$

In order to investigate the error terms here, and to make the notation easier, we study a rescaling limit via a new parameter $y$, where  $y=n\theta$, so that  $y\in [0, c]$.
By applying Hilb's asymptotic on $[0,  c /n]$, we have   
$$P_n(\cos(y/ n))\  =\ J_0(y)+O(n^{-1}) .$$
Next, for,  the rescaling of $P'_n(\cos \theta)$,  we note the relation \cite{AH, CMW}
$$P_n'(\cos\theta)=\frac{n+1}{\sin^2\theta}\,\left[\cos\theta P_n(\cos\theta)-P_{n+1}(\cos\theta)\right].
$$
 %$$\sin\theta P_n'(\cos\theta)=\frac{n+1}{\sin\theta}[\cos\theta P_n(\cos\theta)-P_{n+1}(\cos\theta)],$$ 
Again applying Hilb's asymptotic,  we rescale  ${ [P'_n(\cos\theta)\sin\theta]^2}/{P_n'(1)}   $ to obtain

$$ \frac{[\frac{n+1}{\frac {y}{n}+O(n^{-3})}]^2[(1+O(n^{-2}))(J_0(y+\frac y{2n})+O(n^{-2}))-(J_0(y+\frac {3y}{2n})+O(n^{-2}))]^2}{P_n'(1)}.$$
We apply the Taylor expansion
$$J_0(y+\frac y{2n})=J_0(y)+\frac y{2n}J'_0(y)+O(n^{-2})$$
to further get the rescaling
$$2[J_0'(y)]^2+O(n^{-1}).$$
Hence, as $n\to\infty$, $I_n$ is  asymptotic to
\beq
\label{eq:I}
I_\infty\ =\ \inf_{y\in [0,c]} \frac{1-J_0(y)}{\sqrt{2-2J_0(y)-2[J'(y)]^2 }}.  
\eeq
For $II_n$, we also apply the rescaling technique, the only difference between this and the previous case being in  the estimates of the error terms, where we need to show that the leading terms in the rescaling limits will dominate the error terms. The details are as follows.

Again, take $y=n\theta$, so that now  $y\in [c, n^{1/4}]$.  Hilb's asymptotic gives
$$
P_n(\cos (y/n))=(1+O(n^{-3/2}))^{1/2}J_0(y+{y}/{(2n)})+O(n^{-{15}/8}),  $$
where the uniform bound $O(n^{-{15}/8})$ is achieved when $R_n(\theta)$ is evaluated at $\theta = n^{-3/4}$.

A Taylor expansion yields
$$
J_0(y+\frac y {2n})\ =\  J_0(y)+\frac y{2n}J_0'(y)+\cdots .$$
We now  need two basic properties from \cite{T}  for Bessel functions. The first is that
\begin{equation}\label{J}
J_n(x)
\sim \sqrt{\frac 2{\pi x}}\cos(x-\frac{n\pi}2-\frac\pi 4),\qquad \text{as}\ x \to\infty. 
\end{equation}
The second is that 
\begin{equation}\label{j}J'_0(x)=-J_1(x).\end{equation}
Combining these two properties, we have, for $n$ large enough,  the following uniform  estimate for $y\in[c, n^{1/4}]$;
$$J_0(y+\frac y {2n})=J_0(y)+O(n^{-3/4}).$$
Note that the leading term $J_0(y)$ will always dominate the error term, since, by \eqref{J}, the growth of $J_0(y)$
is at least of order $O(n^{-1/8})$. 

Hence, we have the rescaling limit
$$P_n(\cos \frac yn)= J_0(y)+O(n^{-3/4}) $$
on the interval $y\in [c, n^{1/4}]$.
 
A similar argument shows that the rescaling of  $[P'_n(\cos\theta)\sin\theta]^2/{P_n'(1)}$, for $n$ large enough,
will be dominated by the leading term
$2[J_0'(y)]^2$.  Hence,  $II_n$ will converge, as $n\to\infty$ to  
\beq
\label{eq:II}
II_\infty\ =\ \inf_{y\in [c, \infty]} \frac{1-J_0(y)}{\sqrt{2-2J_0(y)-2[J_0'(y)]^2 }}.
\eeq

We now turn to  $III_n$, which is the last of the three terms to estimate. 
From the the asymptotic expansion \eqref{J},  we see that $J_0((n+\smallhalf)\theta)$ decays rapidly on $\theta\in[n^{- {1}/4},   \pi/2]$, and has, in fact, a uniform bound of $O(n^{-  3/8})$. Thus by the Hilb asymptotic, the same is true of $P_n(\cos\theta)$.  As for  the derivative, Lemma 9.3 of \cite{CMW}  proves 
that, for $\theta\in [c/n, \pi /2]$, 
$$
P_n'(\cos\theta)\ =\ \sqrt{\frac 2\pi}\frac{n^{1/2}}{\sin^{\frac 32}\theta}\left[\sin\phi^--\frac 1{8n\theta}\sin\phi^+\right]+O(n^{-1/2}\theta^{-5/2}),
 $$
where $\phi^{\pm}=(n+\frac 12)\theta\pm \pi/4$.
This implies the rapid decay of ${[P'_n(\cos\theta)\sin\theta]^2}/$ $ {P_n'(1)}$ if we   apply the expression of $P'_n(1)$, and so the $n\to\infty$ limit of 
$III_n$  is
\beq
\label{eq:III}
III_\infty\ =\ \frac1{\sqrt 2}.
\eeq 
Now fix (small) $\epsilon >0$. Combining  \eqref{eq:I}, \eqref{eq:II} and \eqref{eq:III} with \eqref{eq:3parts} and the definition \eqref{eq:rcn} of $r_{c,n}$, 
%we can choose some $\epsilon$ small enough, 
it follows that  there exists a finite  $n_\epsilon$ such that,  for all  even $n>n_\epsilon $, we have
\beq
\label{eq:rcnbound}
r_{c,n}   
\ \geq\  \min\left\{ \inf_{y\in [0,\infty]} \frac{1-J_0(y)}{\sqrt{2-2J_0(y)-2[J_0'(y)]^2 }}   , \frac1{\sqrt 2}\right\}-\epsilon.
\eeq
%Actually, this is the global infimum for the critical radius of $i_n(S^2)$ if $n$ is even by the symmetry of $P_n(\cos\theta)$ about $\pi/2$.
 As an aside, note that if we write the expansion of the Bessel function $J_0(y)=1-{y^2}/4+{y^4}/{64}+O(y^5)$ around $y=0$, then the expression 
 $$\frac{1-J_0(y)}{\sqrt{2-2J_0(y)-2[J_0'(y)]^2}}
 $$
  has the limit, as $y\to 0$,  of $\sqrt{2/3}$. Since this is trivially positive, and $\epsilon$ was arbitrary,   Theorem \ref{1} is now proven for $d=2$ and for even $n$, large enough.
 
However, we still need to treat the cases when $n$ is odd. On the interval $\theta \in [0,\pi/ 2]$,  exactly the same argument as above  for the even case applies, and the same infimum is achieved.  But when we consider on $\theta\in [\pi /2, \pi]$, there is a sign change in the expression of $P_n(\cos \theta)$, since $P_n(-x)=-P_n(x)$ for $n$ odd. 
Taking this into account, we obtain the global lower bound, for arbitrary $\epsilon$ and for all $n$ large enough, of 
\beq
\label{eq:lb2}
%\\   \nonumber
r_{c,n} &\geq&
\min\left\{ \inf_{y\in [0,\infty]} \frac{1-J_0(y)}{\sqrt{2-2J_0(y)-2[J_0'(y)]^2 }}   ,\  \frac1{\sqrt 2},\  
\right.  \\
&& \left.   \qquad\qquad\qquad\qquad \inf_{y\in [0,\infty]} \frac{1+J_0(y)}{\sqrt{2+2J_0(y)-2[J_0'(y)]^2 }} \right\}-\epsilon,
\nonumber
\eeq
 and the proof of Theorem \ref{1} for the case $d=2$ is done.

Figure \ref{fig:test} shows the behaviour of the first and third terms in the lower bound for $r_{c,n}$ in the above inequality.
% \begin{figure}
%\centering
%\begin{subfigure}{0.5\textwidth}
%  \centering
%  \includegraphics[width=.70\linewidth]{graph.eps}
%  \caption{First term}
%  \label{fig:sub1}
%\end{subfigure}%
%\begin{subfigure}{.5\textwidth}
%  \centering
%  \includegraphics[width=.70\linewidth]{2.eps}
%  \caption{Third term}
%  \label{fig:sub2}
%\end{subfigure}
%\label{fig:test}
%\end{figure}

 \begin{figure}
%\centering
%\begin{subfigure}{0.5\textwidth}
 % \centering
  \includegraphics[width=.350\linewidth]{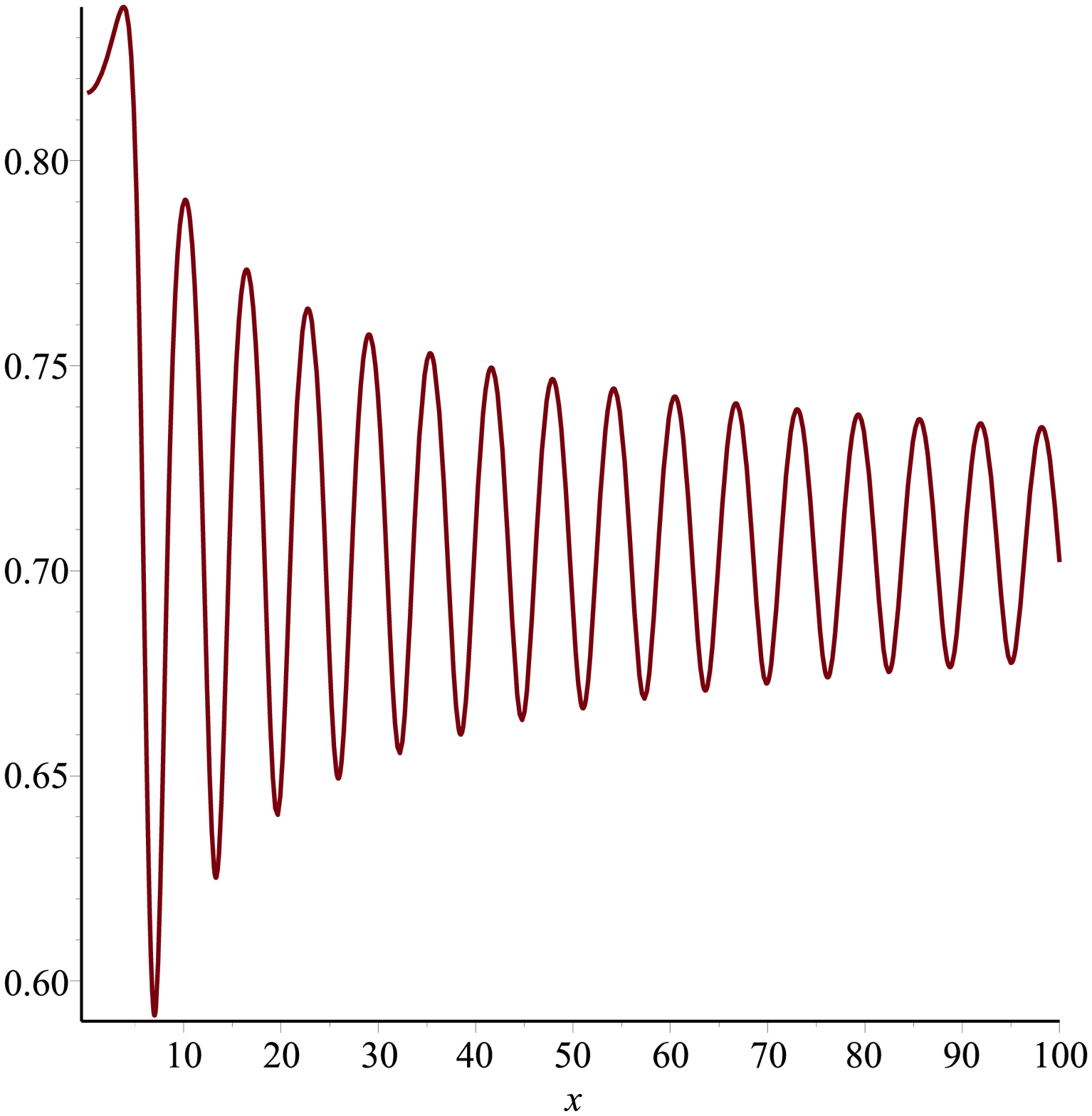}
  \hspace{0.5truein}
%  \caption{First term}
%  \label{fig:sub1}
%\end{subfigure}%
%\begin{subfigure}{.5\textwidth}
%  \centering
  \includegraphics[width=.350\linewidth]{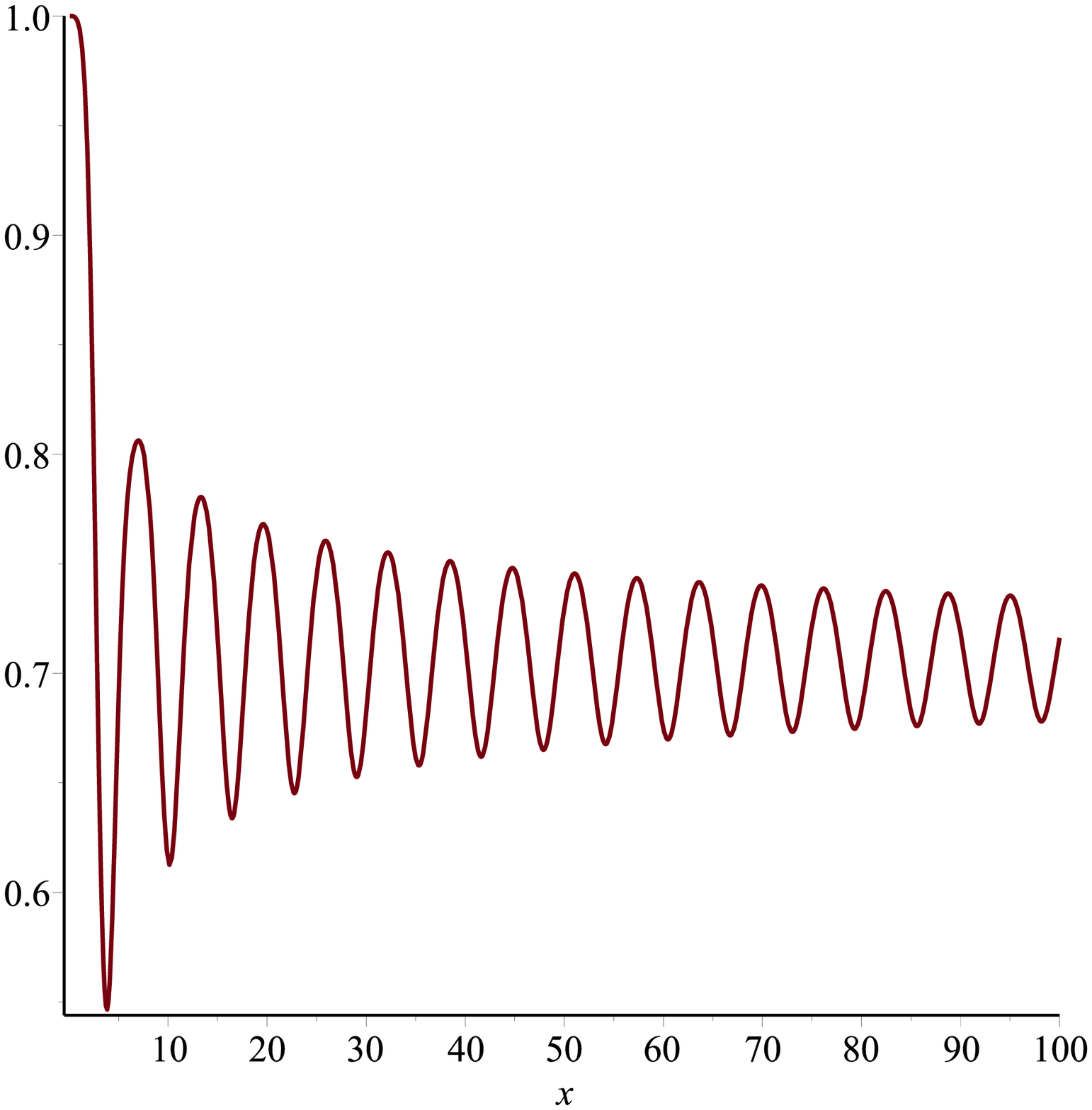}
%  \caption{Third term}
%  \label{fig:sub2}
%\end{subfigure}
\caption{Behaviour of the first and third terms in the lower bound for $r_{c,n}$ in the 2-dimensional case.}
\label{fig:test}
\end{figure}

 \section{Proof of Theorem \ref{1} for the general case}\label{eeee}
The proof of Theorem \ref{1} for two dimensions  can be generalized to higher dimensions without much difficulty.  It relies on properties of spherical harmonics in high dimensions that parallel those of the two dimensional case, and then some heavier notation. (The notation was the main reason for handling the two dimensional case first.) We shall sketch the main arguments in  the proof now.

Retaining the earlier notation, we need to define the
%
%As a first step, we need to recall some basic facts about  spherical harmonics on $S^d$.  Write  $\mathcal H^d_n$ for the space of spherical harmonics of level $n$. Then $\mathcal H^d_n$ is spanned by eigenfunctions of
%$$\Delta_{g_{S^d}}\phi_{j}^{n,d}=-n(n+d-1)\phi_{j}^{n,d},$$
% and the dimension of $\mathcal H^d_n$ is, 
% $$k^d_n:= \dim \mathcal H^d_n=\frac{2n+d-1}{n+d-1}{n+d-1\choose d-1} .$$ 
%Consider the immersion 
%$$i_n^d:\,\,\, S^d\to \mathbb R^{k^d_n},\,\,\, x\to \sqrt{\frac{s_d}{k^d_n}}(\phi_{1}^{n,d},\cdots, \phi_{k^d_n}^{n,d}).$$
%and define the 
normalized spectral projection kernel 
\beqq
\Pi_n^d(x,y) &=& {\frac{s_d}{k^d_n}}\sum_{j=1}^{k^d_n}\phi_{j}^{n,d}(x) \phi_{j}^{n,d}(y), \\
&=&P_{n,d}(\Theta(x,y)), 
\eeqq
the second line following from \cite{AH}.
%, we have $P_{n,d}$ is the $n$-th Legendre polynomials of order $d$ and $\Theta(x,y)$ is the angle between $x$ and $y$.  

%By the properties of the Legendre polynomials, $i_n^d$ is an embedding for $n$ odd and $i_n^d(S^d)\cong \mathbb RP^d$ for $n$ even. We refer to \cite{AH} for more properties regarding $P_{n,d}(x)$. Since $P_{n,d}(1)=1$, we also have $\|i_n^d(x)\|=\Pi_n^d(x,x)=1$. Thus $i_n^d$, 
%\begin{equation}i_n^d:\,\,\, S^d\to  S^{k_n^d-1}\end{equation} 
%is actually a mapping between spheres.

Following the same arguments as those that led to and follow from \eqref{e1}--\eqref{e3}, the pullback of the Euclidean metric is 
\begin{equation}\label{metricback}(i_n^d)^*(g_E)=P_{n,d}'(1)g_{S^d}=\frac{n(n+d-1)}d g_{S^d}.\end{equation}
and the  critical radius of $i_n^d(S^d)$, as a subset of $\real^\kdn$, is exactly the same as before, viz.\ as given by \eqref{eq:rcn}.
Once again, relying on rotation invariance, it suffices to study the local critical radius at  the image of the point $(0,0,\cdots, 1)$.

As before, moving to polar coordinates on $S^d$ we have   %$(x_1,\cdots, x_{d+1})=(\cos\theta)$
 $$\Pi_n^d((0,\cdots,1),y)=P_{n,d}(\cos\theta), \,\,\,\,\,\theta\in [0,\pi].$$
Taking derivatives of the normalized kernel and evaluating them at $(0,\cdots, 1)$, the $d$-dimensional analogue of   \eqref{ppp} now reads 
\begin{equation}\label{ppppppp}\|p_{(0,\cdots,1)}(y)\|^2\ =\ \frac{[\d_{\theta_y}\Pi_n((0,\cdots,1),y)]^2 }{ \d_{\theta_x}\d_{\theta_y}\Pi_n(x,y)|_{x=y}}\ =\  \frac{[P_{n,d}'(\cos\theta)\sin\theta ]^2}{P_{n,d}'(1)} .   \end{equation}
Hence, we can rewrite \eqref{higherg}, now for the  critical radius of the higher dimensional immersion, as
\begin{equation}
\label{eq:rcnd}
r_{c,n}^d(S^d)=\inf_{\theta\in [0,\pi]} \frac{1-P_{n,d}(\cos \theta)}{\sqrt{2-2P_{n,d}(\cos\theta)-\frac{ [P_{n,d}'(\cos\theta)\sin\theta]^2}{P_{n,d}'(1) }}}.\end{equation}
We still have the following Hilb's asymptotic \cite{So}, 
\beqq
P_{n,d}(\cos\theta )
 &=&\Gamma\left(\frac d 2\right)\left[\frac12\left(n+\frac{d-1}2\right)\sin\theta\right]^{-\frac d2 +1}\left(\frac\theta{\sin\theta}\right)^{1/2}
\\
&& \qquad\qquad  \times  J_{\frac d 2 -1}\left(\left(n+\frac{d-1}2\right)\theta\right)\
+\ R_n(\theta),
\eeqq
where \begin{equation}
R_n(\theta)=
\begin{cases}
\theta^{d/2} O(n^{d/2-2}),   &0\leq \theta\leq   c/n\\
 \theta^{1/2}O(n^{-  3/2}),  &c/n\leq \theta \leq   \pi /2,
\end{cases}
\end{equation}
with  $c$ a large, $d$-dependent, constant, and  where $J_{\frac d 2-1}(\theta) $ is the %$(\frac d2-1)$-th
 Bessel function
$$J_{\frac d 2-1}(\theta)=\sum_{j=0} ^\infty \frac{(-1)^j}{j!\Gamma(j+\frac d 2)}\left(\frac \theta 2\right)^{2j+\frac d 2-1}.$$

Again, following the  arguments  of the preceding section,  the global infimum is derived by considering in the subintervals $[0,c/n]$, $[c/n,n^{-3/4}]$ and $[n^{-3/4}, \pi/2]$. The infimum on the first two subintervals is expressed by the  rescaling limit of the Hilb's asymptotic of  the  Legendre polynomials $P_{n,d}(\cos\theta)$. When we rescale $\theta\to y/n$ in the Hilb's asymptotic, we obtain the limit
$$J_\infty^d(y)\ =\ \Gamma(\frac d 2)(\frac12 y)^{-\frac d2 +1} J_{\frac d 2 -1}( y)\ =\  \sum_{j=0} ^\infty \frac{(-1)^j\Gamma({\frac d 2})}{j!\Gamma(j+\frac d 2)}(\frac y 2)^{2j } $$
as $n\to \infty$. On the remaining  subinterval, $[n^{-3/4}, \pi/2]$, the  rapid decay  of $P_{n,d}(\cos\theta)$ and its derivative follow from standard properties of Bessel functions,  (see \eqref{J}), thus the infimum on this subinterval will tend to $1/{\sqrt 2}$, as $n\to\infty$. 

As before, combining arguments for the two cases  of $n$ add and $n$ even, 
%
%Hence, we can conclude that if $n$ is even, the critical radius is asymptotic to $$\inf\{ \inf_{[0,\infty]} \frac{1-J_\infty^d(y)}{\sqrt{2-2J_\infty^d(y)-d[(J_\infty^d)'(y)]^2 }}   , \frac1{\sqrt 2}\};$$
%and if $n$ is odd, the
we find the following lower bound for the critical radii $r_{c,n}^d$ as $n\to\infty$:
%critical radius is asymptotic to 
\beq
\label{eq:lbd}
&&\min\left\{ \inf_{y\in [0,\infty]} \frac{1-J_\infty^d(y)}{\sqrt{2-2J_\infty^d(y)-d[(J_\infty^d)'(y)]^2 }}   , \frac1{\sqrt 2}, 
\right.  \\ &&\qquad\qquad\qquad\qquad \left. \inf_{y\in [0,\infty]} \frac{1+J_\infty^d(y)}{\sqrt{2+2J_\infty^d(y)-d[(J_\infty^d)'(y)]^2 }}\right\},
\nonumber
\eeq
which completes the proof.
%that there is alway a uniform lower bound which only depends on the dimension of the sphere for the immersion of $i_n^d(S^d)$ as $n$ large enough. 

\section{Proof of Theorem  \ref{34}
and  Proposition \ref{prop:2}}
\label{ECTF}
We  break the proofs into three parts, starting  with the proof of  Theorem \ref{34}.

\subsection{The equivalence of mean Euler characteristics and exceedence probabilities}

%In this section, we shall use the tube formula for $i_n^d(S^d)\subset S^{k_n^d-1}$ to derive the formula of the expected Euler characteristic for the spherical harmonics under the spherical ensemble $(S\mathcal H^d_n, \mu^d_n)$. 

The following lemma implies Theorem \ref{34}. It also sets up the relationship between exceedence probabilities and mean Euler characteristics, which we then evaluate in the following two subsections.

% Regarding the Euler characteristic and the tail probability, we have the following, 
\begin{lemma}\label{lemmma}
Under the conditions of Theorem \ref{34}, and for all 
% We can choose a much smaller constant $\rho_d>0$ such that %
%
%for the spherical ensemble $(S\mathcal H_n^d, d\mu_n^d)$, we have, for 
$0\leq \rho\leq \rho_d$, 
\beq\label{equ}
%\mathbb E_{\mu^d_n}\left\{\chi\left\{x\in S^d: \frac{|\Phi^d_n(x)|}{\sqrt {{k^d_n}/{s_d}}}>\cos\rho\right\}\right\}
\E_{\mu_n^d}\left\{\chi\left(A_n^d(\sqrt {\frac{k^d_n}{s_d}} \cos\rho) \right)\right\}
&=&\frac1\kappa \mathbb P_{\mu^d_n} \left\{\sup_{S^d} \frac{\Phi^d_n(x)}{\sqrt {{k^d_n}/{s_d}}}>\cos\rho\right\}\\
&=& \frac{ V_{S^{k_n^d-1}}({\rm Tube}(i_n(S^d), \rho))}{\kappa s_{k_n^d-1}},   \nonumber
%\end{split}
\eeq
where $\kappa$ is $1/2$ if $n$ is even and $1$ if $n$ is odd. 
\end{lemma}
\begin{proof}

We start by noting that by \eqref{eq:pL2} 
we can write
$$\frac{\Phi^d_n(x){}}{\sqrt {{k^d_n}/{s_d}}}\ =\langle a, i^d_n(x)\rangle  =\ \cos \Theta(a, i^d_n(x)),$$
with, as before $a=(a_{1},\cdots, a_{k^d_n})\in S^{k_n^d-1}$, and where $\Theta (x,y)$   is the angle between vectors $x,y\in S^{k_n^d-1}$.

We now note the fact (e.g.\ \cite{S},  Lemma 3.1) that if $M$ is a compact submanifold of a smooth manifold $N$, and $p\in N$, then the intersection between $M$ and a ball of radius $\rho$ around $p$  will either be empty or contractible, as long as $\rho$ is less than the reach of $M$.

Further, we know from Theorem \ref{1} that there is a uniform lower bound for the critical radius of the immersion $i_n^d(S^d)$ in $\mathbb R^{k_n^d}$. From this and a little spherical geometry it follows that the same is true, albeit with a different lower bound, for the critical radius of   $i_n^d(S^d)$ considered as a subset of $S^{k_n^d-1}$. Let $\rho_d$ denote this new lower bound.

Putting the last three paragraphs together, with $M=i_n^d(S^d)$ and $p=a\in N=S^{k_n^d-1}$, we have that 
 the set 
$$
\left\{i_n^d(z)\in S^{k_n^d-1}:\,\,   \langle a, i_n^d(z)\rangle>\cos\rho\right\}
$$
 is either empty or contractible for $0\leq \rho\leq \rho_d$. Hence, 
\beq
\label{equ2}
\kappa \E_{\mu_n^d}\left\{\chi\left(A_n^d( \sqrt {{k^d_n}/{s_d}}\cos\rho) \right)\right\} &=& \mathbb E_{\mu^d_n}\left\{\chi\left\{i_n^d(z)\in S^{k_n^d-1}:\langle a, i_n^d(z)\rangle>\cos\rho\right\}\right\}   \nonumber
\\ &=& \mathbb P_{\mu^d_n} \left\{\sup_z \langle a, i_n^d(z)\rangle >\cos\rho\right\},
\eeq
the factor of $\kappa$ on the right hand side coming from the fact that while $i_n^d$ is an embedding if $n$ is odd,  it identifies  antipodal points if $n$ is even. Consequently,  the Euler characteristic of the preimage on $S^d$ where  will be double that of the image when $n$ is even. This, obviously completes the proof of the lemma.
\end{proof}

As an aside, we note that \eqref{equ2} is also proven in \cite{TK}, although there the approach is to obtain expressions for the  the expected Euler characteristic and the probability separately, and then note  that they are identical.

 %$\kappa$ appears since the map $i_n^d$ is embedding of $S^d$ if $n$ is odd and identifies the antipodal of $S^d$ for $n$ even. 
\subsection{On tube formulae}
Returning to \eqref{equ2}, and noting that, under the spherical ensemble, $a$ is chosen uniformly on $S^{k_n^d-1}$, we have that we can write the final probability there as
\beq
\label{eq:PV}
\mathbb P_{\mu^d_n} \left\{\sup_z \langle a, i_n^d(z)\rangle >\cos\rho\right\} = \frac{ V_{S^{k_n^d-1}}({\rm Tube}(i_n(S^d), \rho))}{s_{k_n^d-1}},
\eeq
where, with a slight -- but space saving -- change of notation, $V_{S^N}$ is volumetric measure with respect to the round metric on $S^N$,   
\beq
{\rm Tube}\left(i_n(S^d), \rho\right)
\ \definedas \
\left\{ x\in S^{k^d_n-1} :\ \min_{y\in i_n(S^d)} d(x,y) \leq\rho\right\},
\eeq
and $d(x,y)$ is geodesic distance on the sphere. 

We now want to express the volume of the tube in \eqref{eq:PV} via Weyl's  tube formula  \cite{AT, G, W}, and so spend the remainder of this section setting up some notation and facts.

 Given an 
$m$-dimensional  Riemannina submanifold $(M,g)$ of  $S^{N-1}$, the volume of  a tube around $M$ of radius $\rho$ less than its critical radius,   is given by 
%
% $(M,g)$ be a Riemannian manifold of $m$ dimension in the unit sphere $S^{N-1}$, then 
(Theorem 10.5.7 in \cite{AT}), 
\begin{equation}\label{tubesphere2} V_{S^{N-1}}(\mbox{Tube}(M, \rho))=\sum_{j=0}^m f_{N,j}(\rho) \mathcal L_j(M)
\end{equation} where \begin{equation}\label{funtions} f_{N,j}(\rho)=\sum_{k=0}^{[\frac j2]}(-4\pi)^{-k}\frac 1{k!}\frac{j!}{(j-2k)!}G_{j-2k,N-1+2k-j}(\rho)
 \end{equation}
 and
  \begin{equation}G_{a,b}(\rho)=\frac{b\pi^{b/2}}{\Gamma(\frac b2+1)}\int_0^\rho \cos^a(r)\sin^{b-1}(r)dr. \end{equation}
The Lipshitz-Killing curvatures $\mathcal L_j(M)$ are given by
\begin{equation}\label{curvatures-old} \mathcal L_j =\begin{cases}
\frac{(-2\pi)^{-(m-j)/2}}{(\frac{m-j}2)!}\int_M \mbox{Tr}(R^{(m-j)/2})\, dV_g, \,\,\,\,\,\,m-j\,\,\mbox{even}\\
0,\,\,\,\,\,\,\,\,\,\,\,\,\,\,\,\, \,\,\,\,\,\,\,\,\,\,\,\,\,\,\,\,\,\,\,\,\,\,\,\,\, \,\,\,\,\,\,\,\,\,\,\,\,\,\,\,\,\,\,\,\,\,\,\,\,\,\,\,\,\,\,\,\,\,\,\,\,\,\,\,\,\,\,\,\,\,\,\,\, m-j\,\,\mbox{odd,}
\end{cases}
\end{equation}
where $R$ is the curvature tensor. In general, $\lips_m(M)=V_g(M)$ is the volume of $M$ and $\lips_0(M)=\chi(M)$ is its Euler characteristic.

For two dimensional surfaces of volume $V_g(M)$ and Euler characteristic $\chi(M)$, embedded in $S^{N-1}$, the tube formula
simplifies to 
\beq
\label{tititi} &&V_{S^{N-1}}(\mbox{Tube}(M, \rho))=\frac{2\pi^{(N-3)/2}}{\Gamma(\frac{N-3}2)}
\\
&&\times \int_0^\rho\sin ^{N-4}(r)\left\{ V_g(M)\left(1-\frac{N-2}{N-3}\sin^2(r)\right)+\frac{2\pi\chi(M) \sin^2 (r)}{N-3}\right\} \, dr.
\nonumber
\eeq

%
%Given a tube of small radius $\rho$ around a compact smooth Riemannian manifold $(M, g)$ of dimension $m$ in Euclidean space $\mathbb R^N$, it's observed by Weyl that all coefficients in the tube formula are independent of the particular way  the submanifold $M $ is embedding, he proved this by expressing the coefficients as integrals of certain complicated curvature functions as follows, 
%
%\begin{equation}\label{tubesphere} V_{\mathbb{R}^N}({\rm Tube}(M, \rho))=\sum_{j=0}^m \rho^{N-j}  \omega_{N-j}\mathcal L_j(M)
%\end{equation}
%where $\omega_{N-j}$ is the volume of the unit ball in $\mathbb R^{N-j}$ and  $\mathcal L_j$ is the Lipschitz-Killing curvature

One final fact that we shall need for later is the value of the Lipshitz-Killing curvatures for spheres. These are
\beq
\label{curvatures}
 \mathcal L_j\left(S^{N-1}\right) =\begin{cases}
2{N-1 \choose j}\frac{s_N}{s_{N-j}},
 &N-1-j\,\,\mbox{even,}\\
0,  &N-1-j\,\,\mbox{odd.}
\end{cases}
\eeq

\subsection{Proof of Proposition \ref{prop:2}.}
The proof works by applying the tube formula 
 \eqref{tubesphere2}  to the equivalence \eqref{eq:PV}.

We tackle the notionally easier case for $S^2$ first, thus proving Corollary \ref{tjm2} directly.  
Then by \eqref{tititi},  for  the  surface $i_n(S^2)$ in the ambient space $S^{2n+1}$, we have  
\beq
\label{eeddd}&&
%\frac{
V_{S^{2n}}(\mbox{Tube}(i_n^d(S^d), \rho))/s_{2n}\\  && \nonumber  
\   = 
\left(
\frac{{2\pi^{n-1}}/{\Gamma(n-1)}}{{2\pi^{n+\frac 12}}/{\Gamma(n+\frac12)}}
\right) %\\ && \times
 \int_0^\rho\sin ^{2n-3}(r)\left\{ V(i_n(S^2))\left(1-\frac{2n-1}{2n-2}\sin^2r\right)
 \right. \\ &&\qquad\qquad\qquad\qquad \qquad\qquad\qquad\qquad\qquad  \left. 
 +\frac{2\pi\chi(i_n(S^2)) \sin^2 (r)}{2n-2}\right\} \,dr.   \nonumber
\eeq
Recall (cf.\  \eqref{metric})  that  the pullback of the Euclidean metric is $((n^2+n)/2)g_{S^2}$. If we combine this with the fact that $i_n(S^2) \cong S^2$ for $n$ odd and $i_n(S^2) \cong \mathbb{R}P^2$ for $n$ even, we have 
\beqq
V(i_n(S^2))&=&2(n^2+n)\pi,\qquad \chi(i_n(S^2))\ =\ 2,\qquad \mbox{for odd}\ n,\\
V(i_n(S^2))&=&(n^2+n)\pi, \qquad\ \ \chi(i_n(S^2))\ =\ 1,\qquad  \mbox{for even}\ n.
   \eeqq
  Substituing this into  \eqref{equ} and noting \eqref{eeddd} suffices to prove   
   Corollary \ref{tjm2}.

For the general, higher dimensional cases,   \eqref{metricback} gives us that 
$$(i_n^d)^*(g_E)=P_{n,d}'(1)g_{S^d},$$
which implies that the curvature tensor of the pullback $(i_n^d)^*(g_E)$ is $[P_{n,d}'(1)]^{-1}R_{g_{S^d}}$ where $R_{g_{S^d}}$ is the curvature tensor of the round metric $g_{S^d}$. Thus $R^{(d-j)/2}(S^d)$ is rescaled to be  
$$[P_{n,d}'(1)]^{(j-d)/2}R^{(d-j)/2}(S^d),
$$
 and the volume form is rescaled to 
 $$
 dV_{(i_n^d)^*(g_E)}=\kappa [P_{n,d}'(1)]^{d/2}dV_{g_{S^d}},
 $$
  where a factor of $\kappa$ appears since the measure on $\mathbb RP^d$ induced from $S^d$ is half of that on $S^d$.  Hence, by definition of the Lipschitz-Killing curvatures in \eqref{curvatures},  the $j$-th Lipschitz-Killing curvature of the pullback metric which involves the integration on $i_n^d(S^d)$ will be rescaled to be  $\kappa[P_{n,d}'(1)]^{j/2}\mathcal L_j(S^d)$.
  Consequently,  
%is $\frac 1{c_d(n)}R_{d}$.
%for the round metric $g_{S^d}$ on the unit sphere $S^d$, the curvature tensor is the Kulkarni-Nomizu product
 %\begin{equation}\label{curvature}R_d=\frac 12g_{S^d} {~\wedge\!\!\!\!\!\!\bigcirc~} g_{S^d}.\end{equation}   Let's denote $K_{2j}(S^d)$ as the $2j$-th integrated mean curvature of $R_d$ defined by \eqref{mc}. Hence, by Weyl's tube formula \eqref{tube2} with $l=d_n$ and $m=d$, we have
\begin{equation}\label{halfd}\frac{V_{S^{k_n^d-1}}(\mbox{Tube}(i_n^d(S^d), \rho))}{s_{k_n^d-1}}=\frac {\kappa} {s_{k_n^d-1}}\sum_{j=0}^d f_{k_n^d,j}(\rho) [P_{n,d}'(1)]^{j/2}\mathcal L_j(S^d),\end{equation}
which, on combining \eqref{equ} and \eqref{halfd}, completes the proof of Theorem \ref{34}.

\section{Some closing comments}
\label{CLOSING}
To conclude, we want to connect our results to some other recent ones, as well as pointing out some interesting open questions.

Given a Riemannian manifold $M$,  \cite{AKTW}
studied the random map 
\begin{equation}\label{robert}i_k: M\to \mathbb R^k,\,\,\,\, x\to k^{-1/2}\left(f_1, f_2,\dots, f_k\right)\end{equation}
where the $f_j$ were independent and identically distributed  copies of  a smooth,  mean zero, unit variance, Gaussian process  $f$. For $k$ large enough, the $i_k$ become embeddings. It was shown that, as $k\to\infty$,  the critical radius of the embedded manifold $i_k(M)$  converged, almost surely,  to a constant known from Gaussian excursion theory, and which depended  on a  Riemannian metric on $M$  induced by the Gaussian process $f$. 

Consider an analogue  of \eqref{robert} in which we replace $f$ by  Gaussian  spherical harmonics on $S^d$ of level $n$. 
That is, we take for $f$ the $\Pdn$ in the form of \eqref{eq:rsh}, but with the $a_j$  standard normal variables.  Note that, as $n\to\infty$, we lose smoothness, and so leave the setting of \cite{AKTW}. 
%As we show in \S\ref{sec1}, the dimension of the spherical harmonics of level $n$ on $S^2$ is $2n+1$, and let's choose an orthogonal basis as $\{\phi^n_{-n}, \cdots, \phi^n_n\}$. We denote the Gaussian random spherical harmonics as $$f^{(n)}(x)=\sum_{j=-n}^n a_j\phi_j^n,$$
%where $a_j$ are normalized Gaussian random variables with mean $0$ and variance $\frac1{\sqrt{2n+1}}$. Hence, by the properties of the projection kernel of the spherical harmonics, we know that the variance of $f^{(n)}$ is $1$.

Consider the random map
$$i_{k,d}^{(n)}: S^d\to \mathbb R^k,\,\,\,x\to \frac 1{\sqrt k}\left(f_1^{(n)},\cdots, f_k^{(n)}\right),$$
where  the $f_j^{(n)}$ are independent and identically distributed  copies of $\Pdn$. 
When $k$ is large enough, $i^{(n)}_{k,d}$ is still an embedding. However, as opposed to the setting \eqref{robert}, the interesting problem now is  the decay rate of the critical radius of the embedded sphere as $n\to\infty$, but with fixed $k$, large enough. The method used in \cite{AKTW} highly depends on a  central limit theorem  as $k\to\infty$, and so their method is not applicable in this problem. The generic behavior of the critical radius of $i_{k,d}^{(n)}(S^d)$ as $n\to\infty$ is unclear.

Another problem, more closely related to what we have studied here, is to understand the critical radius for more general Riemannian manifolds. That is, given a $d$-dimensional Riemannian manifold $(M, g)$,   consider the eigenspace $$\mathcal H^d_{[\lambda,\lambda+1]}:=\{\phi: \, \Delta_g \phi=-\tilde\lambda\phi, \,\, \tilde \lambda \in [\lambda,\lambda+1] \},$$
for large $\lambda$. Then  choose $\{\phi_1, \cdots, \phi_{k^d
_\lambda}\}$ as the orthogonal basis  of $\mathcal H^d_{[\lambda,\lambda+1]}$
and define the  immersion, 
\begin{equation}\label{ilam}i^d_\lambda:\,\,\, M\to \mathbb R^{k^d_\lambda},\,\,\, x\to \left(k^d_\lambda\right)^{-1/2}\left(\phi_{1},\dots, \phi_{k^d_\lambda}\right),\end{equation}
where $k^d_\lambda$ is the dimension of $\mathcal H^d_{[\lambda,\lambda+1]}$. 

This map is not new, and was considered by Zelditch  in \cite{Z}, for  Zoll and aperiodic manifolds. He  obtained the leading order terms of the spectral projection kernel and its derivatives, from which he was able to derive asymptotics for the distribution of zeros of Gaussian random waves by the classical Kac-Rice formula. 

In the results of the current paper, our computations regarding the critical radius for the  immersion $i_n^d(S^d)$ relied on the fact that all the information of the immersion $i_n^d$ \eqref{id}
is contained in the spectral projection kernels. To be more precise, we needed the leading expansion and the rescaling limit of the spectral projection kernel and its derivatives up to order two. 
It seems that our method can be generalized to the case of Zoll  and aperiodic manifolds.  It is well known that the behavior of eigenfunctions highly depends on the dynamical system of the manifolds \cite{Ze3}, and it should be very interesting to study the relation between the critical radius of $i^d_\lambda(M)$
 and the dynamical system.
 %, for example, does the immersion $i^d_\lambda(M)$ has a uniform lower bound for its critical radius if the eigenfunctions satisfy the quantum ergodicity property?  
% One may test this problem by studying the case of the Riemann surfaces of higher genus equipped with the constant negative scalar curvature which are quantum ergodicity, then compare with the result we obtain in this article for the two dimensional sphere which is a quantum integrable system. 
 We postpone these questions for  further investigation.

\section*{Acknowledgements}

We  are grateful to  Sunder Ram Krishnan for many useful discussions in the early stages of our research.

\end{document}